\documentclass[a4paper,reqno]{amsart}
\usepackage{amsmath,amssymb,amsthm,amsfonts}
\usepackage{enumerate}
\usepackage{esint}
\usepackage[inline,nomargin]{fixme}

\newtheorem{theorem}{Theorem}
\newtheorem{corollary}[theorem]{Corollary}
\newtheorem{lemma}[theorem]{Lemma}
\newtheorem{proposition}[theorem]{Proposition}
\theoremstyle{definition}
\newtheorem{definition}[theorem]{Definition}
\theoremstyle{remark}
\newtheorem{remark}[theorem]{Remark}

\def\reals{\mathbf{R}}

\def\complexes{\mathbf{C}}

\def\E{\mathbb{E}\,} 
\def\1{\mathbf{1}}
\def\F{\mathcal{F}}

\renewcommand{\Re}{\operatorname{Re}}
\renewcommand{\Im}{\operatorname{Im}}

\title{On the multiplicative chaos of non-Gaussian log-correlated fields}
\author{Janne Junnila}
\thanks{The author was supported by the Doctoral Programme in Mathematics and Statistics at University of Helsinki}
\address{University of Helsinki, Department of Mathematics and Statistics, P.O. Box 68, FIN-00014 University of Helsinki, Finland}
\email{janne.junnila@helsinki.fi}


\begin{document}

\maketitle

\begin{abstract}
  We study non-Gaussian log-correlated multiplicative chaos, where the random field is defined as a sum of independent fields that satisfy suitable moment and regularity conditions. The convergence, existence of moments and analyticity with respect to the inverse temperature are proven for the resulting chaos in the full subcritical range. These results are generalizations of the corresponding theorems for Gaussian multiplicative chaos. A basic example where our results apply is the non-Gaussian Fourier series
 \[\sum_{k=1}^\infty \frac{1}{\sqrt{k}}(A_k \cos(2\pi k x) + B_k \sin(2\pi k x)),\]
 where $A_k$ and $B_k$ are i.i.d. random variables.
\end{abstract}

\section{Introduction}

\noindent The theory of multiplicative chaos was originally introduced by Kahane \cite{kahane1985,kahane1987positive} as a continuous analogy of Mandelbrot cascades \cite{mandelbrot1974multiplications}.
Kahane's theory concerns weak$^*$-limits of random measures of the form
\[d\mu_n(x) = e^{\sum_{k=1}^n X_k(x) - \frac{1}{2} \sum_{k=1}^n \E X_k(x)^2} \, d\lambda(x),\]
where $X_k$ are independent centered Gaussian random fields on some metric measure space $T$ and $\lambda$ is a reference measure.
It is easy to see that the measures $\mu_n$ form a martingale and converge in the weak$^*$-sense to a limit measure $\mu$ that we call the \emph{multiplicative chaos} associated to the sequence $X_k$.
However, the first non-trivial question in the theory is whether $\mu$ is almost surely zero or not.

In the Euclidean setting Kahane identified the log-correlated random fields to be the edge case when it comes to the non-triviality of the resulting chaos.
We say that the Gaussian random field $X = \sum_{k=1}^\infty X_k$ is log-correlated if it formally has a covariance of the form
\[\E X(x) X(y) = \beta^2 \log \frac{1}{|x-y|} + g(x,y),\]
where $g$ is a bounded and continuous function and $\beta > 0$ is a constant. The parameter $\beta$ is often called the \emph{inverse temperature} in the mathematical physics literature.
The non-triviality of the chaos measure $\mu$ then depends on $\beta$ and the dimension $d$ of the space: If $0 < \beta < \sqrt{2d}$, $\mu$ is almost surely non-trivial; if $\beta \ge \sqrt{2d}$, it is almost surely zero.
The regime $0 < \beta < \sqrt{2d}$ is called \emph{subcritical}, while the regimes $\beta = \sqrt{2d}$ and $\beta > \sqrt{2d}$ are called \emph{critical} and \emph{supercritical}, respectively.
In the critical and supercritical cases it is still possible to get non-trivial measures by performing a suitable renormalization \cite{duplantier2014critical,duplantier2014renormalization,madaule2016glassy}.

The study of Gaussian multiplicative chaos in the log-correlated case has spurred a lot of interest, and comprehensive reviews exist \cite{rhodes2014gaussian,rhodes2016lecture}.
In this situation there are also results on uniqueness and convergence under different approximations \cite{robert2010gaussian,rhodes2014gaussian,shamov2014gaussian,junnila2015uniqueness}.
In the non-Gaussian case the research has so far focused mainly on infinitely divisible processes.
The paper \cite{bacry2003log} studies chaoses that are defined using a cone construction with an infinitely divisible independently scattered random measure on the upper half plane.
A more recent paper \cite{rhodes2014levy} deals with chaoses that are $\star$-scale invariant, a specific class which again implies infinitely divisibility under some small assumptions.
Finally, in \cite{saksman2016multiplicative} a field obtained from a statistical model of the Riemann $\zeta$-function on the critical line is studied.
The resulting chaos measure in this case is almost surely absolutely continuous with respect to a Gaussian chaos measure.

In this paper we will study the multiplicative chaos of (possibly strongly) non-Gaussian locally log-correlated random fields defined on the closed unit cube $I = [0,1]^d \subset \reals^d$, see Definition \ref{def:log_correlated} below.\footnote{Restricting to the unit cube is done purely on practical grounds and the results generalize easily to other domains.}
Let $X_k$ ($k=1,2,\dots$) be continuous, independent and centered random fields on $I$ and $\beta>0$ be a parameter.
We define a sequence of measures on $I$ by setting
\begin{equation}\label{eq:mu_n_definition}
  \mu_n(f;\beta) = \int_I f(x) \frac{e^{\beta \sum_{k=1}^n X_k(x)}}{\E e^{\beta \sum_{k=1}^n X_k(x)}} \, dx
\end{equation}
for all $f \in C(I,\reals)$.
Since we are not anymore in the Gaussian case, we must make the extra assumption that $\E e^{\beta X_k(x)}$ exists for all $k \ge 1$.
Still, $\mu_n(f;\beta)$ is a martingale for all fixed $f \in C(I,\reals)$ and $\beta > 0$, and one gets the almost sure weak$^*$-convergence as in Kahane's theory. Again the crucial question is whether the limit is non-trivial or not.

We will in fact allow $\beta$ to take complex values, in which case it becomes important to ensure that the denominator $\E e^{\beta X_k(x)}$ does not vanish.
Clearly for any finite $n$ it is possible to pick a neighbourhood of $(0,\sqrt{2d})$ (which will be our region of interest) where this is true. The assumptions that will follow will ultimately ensure that one can choose such a neighbourhood in such a way that this holds for all $n \ge 1$ simultaneously.
It is worth noting that in the case of complex $\beta$ we do not expect the limit $\mu$ to be a random measure but rather a distribution.
In the present paper we will however settle for proving that for a fixed $f \in C(I,\complexes)$ the limit exists almost surely and is analytic with respect to the parameter $\beta$.

As an interesting example of a non-Gaussian locally log-correlated random field, consider the random Fourier series $\sum_{k=1}^\infty X_k(x)$, where
\begin{equation}\label{eq:fourier_field}
  X_k(x) = \frac{1}{\sqrt{k}} (A_k \cos(2\pi k x) + B_k \sin(2\pi k x)), \quad x \in [0,1].
\end{equation}
Here $A_k$ and $B_k$ are i.i.d.\ centered random variables with variance $1$.
In this case we have the following theorem.

\begin{theorem}\label{thm:fourier_result}
  Assume that the fields $X_k$ are defined as in \eqref{eq:fourier_field} and that $\E e^{\lambda A_1} < \infty$ for all $\lambda \in \reals$.
  Then there exists an open set $U \subset \complexes$ containing the interval $(0,\sqrt{2d})$ such that for any compact $K \subset U$ there exists $p > 1$ for which the martingale $\mu_n(f;\beta)$ converges in $L^p(\Omega)$ for all $\beta \in K$ and $f \in C(I,\complexes)$.

  As a corollary, for a fixed $f \in C(I,\complexes)$ the maps $\beta \mapsto \mu_n(f;\beta)$ converge almost surely uniformly on compact subsets of $U$ to an analytic map $\beta \mapsto \mu(f;\beta)$.
\end{theorem}

\begin{remark}
  For $\beta \in \reals$ a standard argument using the fact that $C(I,\reals)$ is separable shows that the maps $\beta \mapsto \mu_n(f;\beta)$ converge almost surely for all $f \in C(I)$ simultaneously, and as a limit one obtains a random measure $\mu(\cdot;\beta)$ which is a continuous function of $\beta$ in the weak$^*$-topology of measures.
\end{remark}

The extension to the complex case is quite non-trivial in this situation, since we are missing local independence of the increments $X_k$, and hence the previously known methods for proving analyticity of the chaos fail completely. Here we develop a new method inspired by the clever and simple new approach due to Berestycki \cite{berestycki2015elementary}. In contrast to \cite{berestycki2015elementary} we completely bypass the $L^1$-estimates, performing instead a direct estimate in $L^p$ via a dyadic analysis of the field. As a further distinction to \cite{berestycki2015elementary}, one may note that Girsanov's lemma is not applicable in the non-Gaussian setting.

Theorem~\ref{thm:fourier_result} is a corollary of a more general result, which we state next.

\begin{definition}\label{def:log_correlated}
We say that the sequence $(X_k)_{k=1}^\infty$ has a locally log-correlated structure if the following hold:
\begin{itemize}
  \item We have $\sup_{x \in I} \E X_k(x)^2 \to 0$ and $\sum_{k=1}^\infty \E X_k(0)^2 = \infty$.

  \item There exists a constant $\delta > 0$ such that for all $n \ge 1$ and $x,y \in I$ satisfying $|x-y| \le \delta$ we have
    \begin{equation}\label{eq:log_correlated}
      |\sum_{k=1}^n \E X_k(x) X_k(y) - \min\Big(\log \frac{1}{|x-y|}, \sum_{k=1}^n \E X_k(0)^2 \Big)| \le C
    \end{equation}
    for some constant $C > 0$.
\end{itemize}
\end{definition}

In addition to having a locally log-correlated structure, we will require certain regularity of the fields $X_k$, which we list as conditions \eqref{eq:moment_condition} and \eqref{eq:continuity}:
\begin{align}
  & \sup_{x \in I} \sum_{k=1}^\infty \Big(\E |X_k(x)|^{3+\varepsilon}\Big)^{\frac{3}{3+\varepsilon}} < \infty \quad \text{for some } \varepsilon > 0 \label{eq:moment_condition} \\
  & \E \left|\sum_{k=1}^n (X_k(x) - X_k(y)) \right|^r \le C_r e^{r \sum_{k=1}^n \E X_k(0)^2} |x-y|^r \quad \text{for } n \ge 1, \text{ and } r \ge r_0 \label{eq:continuity}
\end{align}
In condition \eqref{eq:continuity} the constant $C_r > 0$ may depend on $r$ and $r_0>0$ is an arbitrary finite constant.
We note that \eqref{eq:continuity} can be deduced from either of the following two conditions:
\begin{align}
  \sum_{k=1}^n \E |X_k(x) - X_k(y)|^r \le C_r e^{r \sum_{k=1}^n \E X_k(0)^2} |x-y|^r \quad \text{for } n \ge 1, \text{ and } r \ge 2 \label{eq:continuity2} \\
  \sum_{k=1}^n \E \sup_{x \in I} |X_k'(x)|^r \le C_r e^{r \sum_{k=1}^n \E X_k(0)^2} \quad \text{for } n \ge 1, \text{ and } r \ge 2 \label{eq:continuity3}.
\end{align}
Indeed, the mean value theorem shows that \eqref{eq:continuity3} implies \eqref{eq:continuity2}, while Rosenthal's inequality \cite{rosenthal1970subspaces} shows that \eqref{eq:continuity2} implies \eqref{eq:continuity}.
Note that in condition \eqref{eq:continuity3} we implicitly assume that $X_k(x)$ is almost surely continuously differentiable on $I$.

\begin{theorem}\label{thm:main_result}
  Assume that $(X_k)_{k=1}^\infty$ is a sequence of independent, centered and continuous random fields having a locally log-correlated structure and satisfying \eqref{eq:moment_condition} and \eqref{eq:continuity}. Assume further that
  \begin{equation}\label{eq:X_exponential_moments}
    \sup_{k \ge 1} \sup_{x \in I} \E e^{\lambda X_k(x)} < \infty
  \end{equation}
  for all $\lambda \in \reals$.
  Then the conclusions of Theorem~\ref{thm:fourier_result} hold for the measures $\mu_n$.
\end{theorem}

The proof of Theorem~\ref{thm:main_result} is given in Section 2, and a brief outline is as follows:
We start by normalizing the situation in such a way that the variance of the field on the level $n$ is approximately $n \log(2)$.
For each $x \in I$ we focus on the last level $l$ on which the field is exceptionally large.
For points sharing a common level $l$ we perform a splitting of $I$ into dyadic cubes with side length approximately $2^{-l}$ and derive an $L^2$-estimate in each of these cubes, conditioned on the level $l$.
This takes care of the contribution coming from the tail of the field.
After the conditional estimate we are still left with the contribution coming from level $l$, which is then handled by approximating the $L^p$-norm of the exponential of the supremum of the field, relying on the fact that the field being exceptionally large on level $l$ is an event of low probability.
Once we have established the boundedness in $L^p$, the rest of the proof is rather routine.

In Section 3 we first prove Theorem~\ref{thm:fourier_result}, after which we provide another application of Theorem~\ref{thm:main_result} where we consider a field that is the sum of dilated stationary processes, see Theorem~\ref{thm:general_scaling}.
The latter example is related to the general model presented in \cite{mannersalo2002multifractal}.

\medskip
\noindent \textit{Acknowledgements}.\quad I wish to thank my thesis advisor Eero Saksman for proposing this problem, as well as for the many fruitful discussions regarding it and other topics.
Many thanks also to Christian Webb and Julien Barral for their valuable comments on the paper.

\section{Proof of Theorem~\ref{thm:main_result}}

\noindent We start by splitting the field $\sum_{k=1}^\infty X_k(x)$ into blocks of approximately constant variances.
For all $j \ge 0$ let $t_j \ge 1$ be the smallest index for which
\begin{equation}\label{eq:tj}
  \sum_{k=1}^{t_j} \E X_k(0)^2 \ge \log(2) j.
\end{equation}
The $t_j$ are well-defined because of our assumption that $\sum_{k=1}^\infty \E X_k(0)^2 = \infty$, and we can use them to define the following auxiliary fields
\[Y_j(x) = \sum_{k=t_{j-1}}^{t_j - 1} X_k(x), \quad Z_j(x) = \sum_{k=1}^j Y_k(x) \quad (j \ge 1).\]
Here we use the convention that $Y_j(x) = 0$ for all $x \in I$ if $t_{j-1} = t_j$, and we also set $Z_0(x) \equiv 0$.

The following lemma shows how the locally log correlated structure of the fields $X_k$ transfers to the fields $Z_j$.
\begin{lemma}\label{lemma:z_covariance}
  There exists a constant $C > 0$ such that for $|x-y| \le \delta$ we have
  \begin{equation}\label{eq:z_covariance}
    \Big|\E Z_j(x)Z_j(y) - \min\Big(\log \frac{1}{|x-y|}, \log(2) j\Big)\Big| \le C.
  \end{equation}
  In particular the following inequalities hold for some constant $C > 0$:
  \begin{equation}\label{eq:z_nearby_covariance}
    |\E Z_j(x)Z_j(y) - \log(2) j| \le C \quad  (j \ge 1, \; |x-y| \le 2^{-j})
  \end{equation}
  \begin{equation}\label{eq:z_long_range_covariance}
    |\E Z_j(x)Z_j(y) - \E Z_m(x) Z_m(y)| \le C \quad (j \ge m, \; 2^{-m-1} \le |x-y| \le \delta)
  \end{equation}
\end{lemma}

\begin{proof}
  We have $\E Z_j(x) Z_j(y) = \sum_{k=1}^j \E Y_k(x) Y_k(y) = \sum_{k=1}^{t_j - 1} \E X_k(x) X_k(y)$.
  By \eqref{eq:tj} and the assumption that $\E X_k(0)^2 \to 0$, the variance on the level $t_j - 1$ satisfies
  \[\sum_{k=1}^{t_j - 1} \E X_k(0)^2 = \log(2) j + O(1),\]
  so \eqref{eq:z_covariance} follows from \eqref{eq:log_correlated}.
  The inequalities \eqref{eq:z_nearby_covariance} and \eqref{eq:z_long_range_covariance} are easy corollaries.
\end{proof}

The next lemma provides a crucial estimate on the Laplace transform of the fields $Z_j$ and it will be used extensively in the proofs.
The idea behind it is the following: Since $\E X_k(x)^2$ tends to $0$, the constant variance increments $Y_j$ will start to look like Gaussians by the central limit theorem.
This leads one to expect that some sort of Gaussianity appears also in the fields $Z_j$ and here we quantify this for the Laplace transform of the vector $(Z_j(x), Z_j(y))$, where $x,y \in I$.
\begin{lemma}\label{lemma:z_mgf}
  Let $R > 0$.
  There exists $r = r(R) > 0$ such that if $\xi_1,\xi_2 \in \complexes$ satisfy $|\Re \xi_i| \le R$ and $|\Im \xi_i| \le r$ for $i=1,2$, then 
  \begin{equation}
    \E e^{\xi_1 Z_j(x) + \xi_2 Z_j(y)} = e^{\frac{\xi_1^2}{2} \E Z_j(x)^2 + \frac{\xi_2^2}{2} \E Z_j(y)^2 + \xi_1 \xi_2 \E Z_j(x) Z_j(y) + \varepsilon},
  \end{equation}
  where the error $\varepsilon=\varepsilon(\xi_1,\xi_2)$ is bounded and $|\varepsilon| \le C_R$ for some constant $C_R > 0$ depending on $R$.
\end{lemma}

\begin{proof}
  Let $K = [-R,R] \times [-r,r]$, where $r \in (0,1]$ will be chosen later.
  Define $\varphi_k(\xi_1,\xi_2) = \E e^{\xi_1 X_k(x) + \xi_2 X_k(y)}$.
  Then Taylor's theorem gives us
  \[\varphi_k(\xi_1,\xi_2) = 1 + \frac{\E X_k(x)^2}{2} \xi_1^2 + \frac{\E X_k(y)^2}{2} \xi_2^2 + \E X_k(x) X_k(y) \xi_1 \xi_2 + c_k O(|\xi|^3),\]
  where
  \begin{align*}
    c_k & = \sup_{\xi_1,\xi_2 \in K} |\frac{\partial^3}{\partial \xi_1^3} \varphi_k(\xi_1,\xi_2)| + \sup_{\xi_1,\xi_2 \in K} |\frac{\partial^3}{\partial \xi_2^3} \varphi_k(\xi_1,\xi_2)| \\
          & \quad + \sup_{\xi_1,\xi_2 \in K} |\frac{\partial^3}{\partial \xi_1^2 \partial \xi_2} \varphi_k(\xi_1,\xi_2)| + \sup_{\xi_1,\xi_2 \in K} |\frac{\partial^3}{\partial \xi_1 \partial \xi_2^2} \varphi_k(\xi_1,\xi_2)|.
  \end{align*}
  We have for all $a, b \in \{0,1,2,3\}$, $a+b = 3$, that
  \[|\frac{\partial^3}{\partial \xi_1^a \xi_2^b} \varphi_k(\xi_1,\xi_2)| \le \E |X_k(x)|^a |X_k(y)|^b |e^{\xi_1 X_k(x) + \xi_2 X_k(y)}|,\]
  which by Hölder's inequality is less than
  \[\Big(\E |X_k(x)|^{3+\varepsilon}\Big)^{\frac{a}{3+\varepsilon}} \Big(\E |X_k(y)|^{3+\varepsilon}\Big)^{\frac{b}{3+\varepsilon}} \Big(\E |e^{\frac{3+\varepsilon}{\varepsilon/2} \xi_1 X_k(x)}|\Big)^{\frac{\varepsilon/2}{3+\varepsilon}} \Big(\E |e^{\frac{3+\varepsilon}{\varepsilon/2} \xi_2 X_k(x)}|\Big)^{\frac{\varepsilon/2}{3+\varepsilon}},\]
  where $\varepsilon$ is given by the assumption \eqref{eq:moment_condition}.
  The last two factors are bounded by the assumption \eqref{eq:X_exponential_moments}.
  Finally because
  \begin{align*}
    & \Big(\E |X_k(x)|^{3+\varepsilon}\Big)^{\frac{a}{3+\varepsilon}} \Big(\E |X_k(y)|^{3+\varepsilon}\Big)^{\frac{b}{3+\varepsilon}} \\
    & \le \frac{a \Big(\E |X_k(x)|^{3 + \varepsilon}\Big)^{\frac{3}{3+\varepsilon}} + b \Big(\E |X_k(y)|^{3 + \varepsilon}\Big)^{\frac{3}{3+\varepsilon}}}{3},
  \end{align*}
  we see that
  \begin{equation}\label{eq:c_k_finite}
    \sum_{k=1}^\infty c_k < \infty
  \end{equation}
  by \eqref{eq:moment_condition}.
  Because $\sup_{x \in I} \E X_k(x)^2 \to 0$, there exists $k_0 \ge 1$ such that for large enough $k \ge k_0$ we have
  \[\Big|\frac{\E X_k(x)^2}{2} \xi_1^2 + \frac{\E X_k(y)^2}{2} \xi_2^2 + \E X_k(x) X_k(y) \xi_1 \xi_2 + c_k O(|\xi|^3)\Big| < \frac{1}{2}\]
  whenever $\xi_1,\xi_2 \in K$.
  In particular if $\log \colon \complexes \setminus (-\infty,0] \to \complexes$ is the branch of the logarithm that takes the value $0$ at $1$, we have
  \[\log(\varphi_k(\xi_1, \xi_2)) = \frac{\E X_k(x)^2}{2} \xi_1^2 + \frac{\E X_k(y)^2}{2} \xi_2^2 + \E X_k(x) X_k(y) \xi_1 \xi_2 + d_k O(|\xi|^3 + |\xi|^6).\]
  Here 
  \begin{align*}
    d_k & = c_k + (\E X_k(x)^2)^2 + (\E X_k(y)^2)^2 + (\E X_k(x) X_k(y))^2 + c_k^2 \\
     & \quad + (\E X_k(x)^2 + \E X_k(y)^2 + |\E X_k(x) X_k(y)|) c_k + \E X_k(x)^2 \E X_k(y)^2 \\
     & \quad + \E X_k(x)^2 |\E X_k(x) X_k(y)| + \E X_k(y)^2 |\E X_k(x) X_k(y)|.
  \end{align*}
  By Hölder's inequality and \eqref{eq:moment_condition} we have
  \[\sup_{x \in I} \sum_{k=1}^\infty (\E X_k(x)^2)^2 \le \sup_{x \in I} \sum_{k=1}^\infty (\E |X_k(x)|^{3 + \varepsilon})^{\frac{4}{3+\varepsilon}} < \infty,\]
  and this together with \eqref{eq:c_k_finite} and Hölder's inequality gives $\sum_{k=1}^\infty d_k < \infty$.
  Now if $\varphi(\xi_1,\xi_2) = \E e^{\xi_1 Z_j(x) + \xi_2 Z_j(y)}$, we have by independence that
  \[\varphi(\xi_1, \xi_2) = \prod_{k=1}^{k_0 - 1} \varphi_k(\xi_1,\xi_2) \prod_{k=k_0}^{t_j-1} \varphi_k(\xi_1, \xi_2).\]
  By continuity $r$ can be chosen so small that the absolute value of the first product is bounded from below and from above for all $\xi_1,\xi_2 \in K$, and hence the product can be swallowed into the constant $C_R$.
  The logarithm of the second product is
  \begin{align*}
    & \sum_{k=k_0}^{t_j - 1} \Big(\frac{\E X_k(x)^2}{2} \xi_1^2 + \frac{\E X_k(y)^2}{2} \xi_2^2 + \E X_k(x) X_k(y) \xi_1 \xi_2 + d_k O(|\xi|^3 + |\xi|^6)\Big) \\
    & = \frac{\xi_1^2}{2} \E Z_j(x)^2 + \frac{\xi_2^2}{2} \E Z_j(y)^2 + \xi_1 \xi_2 \E Z_j(x) Z_j(y) - \frac{\xi_1^2}{2} \sum_{k=1}^{k_0 - 1} \E X_k(x)^2 \\
    & \quad - \frac{\xi_2^2}{2} \sum_{k=1}^{k_0 - 1} \E X_k(y)^2 - \xi_1 \xi_2 \sum_{k=1}^{k_0 - 1} \E X_k(x) X_k(y) + O(|\xi|^3 + |\xi|^6)
  \end{align*}
  and everything but the first $3$ terms can be put in $C_R$.
\end{proof}

To prepare for the proof of Theorem~\ref{thm:main_result} we start by fixing some notation and then prove the main estimates as lemmas.
First of all, we assume that $\delta > \sqrt{d}$, so that the estimates \eqref{eq:z_covariance} and \eqref{eq:z_long_range_covariance} hold for all $x,y \in I$. We will show how to get rid of this assumption in the end.
Second, for a given $\beta$ we let $\alpha \in (\Re \beta, 2\Re \beta)$ be a fixed real parameter.
We will not specify the exact value of $\alpha$, but it will be clear from the proof that choosing it sufficiently close to $\Re \beta$ will work.
We assume that $\Re \beta \in (0,\sqrt{2d})$ and that
\begin{equation}\label{eq:beta_imag_restriction}
  (\Im \beta)^2 < \min\Big\{\frac{(\alpha - \Re \beta)^2}{2}, \frac{(2\Re \beta - \alpha)^2}{2} - (\Re \beta)^2 + d, r^2\Big\},
\end{equation}
where $r$ is obtained from Lemma~\ref{lemma:z_mgf} applied with $R = 2\sqrt{2d}$.
Notice that by choosing $\alpha$ close enough to $\Re \beta$ it is always possible to have $\frac{(2 \Re \beta - \alpha)^2}{2} - (\Re \beta)^2 + d > 0$.

For $l \ge 0$ we let
\[A_l(x) = \{Z_l(x) \ge \alpha \E Z_l(x)^2\}\]
be the event that $Z_l$ is exceptionally large at the point $x \in I$.
Notice that by definition $A_0(x)$ happens surely.
Similarly for $k \ge l$ we let
\[B_{l,k}(x) = \{Z_j(x) < \alpha \E Z_j(x)^2 \text{ for all } l+1 \le j \le k\}\]
be the event that $Z_j$ is small from level $l+1$ to level $k$.
We note that $B_{l,l}(x)$ happens surely.

To keep formulas short (or at least shorter), define
\[E_k(x) = \frac{e^{\beta Z_k(x)}}{\E e^{\beta Z_k(x)}}\]
for all $k \ge 1$, together with the notation
\begin{align*}
Z_{(m,k]}(x) & = Z_k(x) - Z_m(x) \\
E_{(m,k]}(x) & = \frac{e^{\beta Z_{(m,k]}(x)}}{\E e^{\beta Z_{(m,k]}(x)}} \\
A_{(m,k]}(x) & = \{Z_{(m,k]}(x) \ge \alpha \E Z_{(m,k]}(x)^2\}
\end{align*}
for all $0 \le m \le k$.
Moreover, define
\[\F_m = \sigma(X_1,\dots,X_{t_m - 1}).\]
Note that the variables with the subscript $(m,k]$ are independent of $\F_m$.

Finally we use the notation $A \lesssim B$ to indicate that there exists a constant $C>0$ only depending on $\alpha$, $\beta$, $d$ and the distribution of the fields $X_k$ such that the inequality $A \le C B$ holds.
We write $A \approx B$ when both $A \lesssim B$ and $B \lesssim A$ hold.

We start with a couple of technical lemmas.
 
\begin{lemma}\label{lemma:tail_field_estimate}
  Let $k \ge m$. Then
  \begin{align*}
& \E[|E_{(m,k]}(x) \overline{E_{(m,k]}(y)}| \1_{A_{(m,k]}(x)}] \\
& \lesssim e^{- \frac{(\alpha - \Re \beta)^2}{2} \log(2) (k-m)} e^{(\Im \beta)^2 \log(2)(k-m) + \alpha (\Re \beta) \E[Z_{(m,k]}(x) Z_{(m,k]}(y)]}
  \end{align*}
\end{lemma}

\begin{proof}
  By Lemma~\ref{lemma:z_mgf} we have
  \begin{align}\label{eq:modulus_of_E_product}
& |E_{(m,k]}(x) \overline{E_{(m,k]}(y)}| = \frac{e^{(\Re \beta) (Z_{(m,k]}(x) + Z_{(m,k]}(y))}}{|\E e^{\beta Z_{(m,k]}(x)}| |\E e^{\overline{\beta} Z_{(m,k]}(y)}|} \\
& \quad \lesssim e^{\frac{(\Im \beta)^2 - (\Re \beta)^2}{2} (\E Z_{(m,k]}(x)^2 + \E Z_{(m,k]}(y)^2)} e^{(\Re \beta) (Z_{(m,k]}(x) + Z_{(m,k]}(y))}. \nonumber
  \end{align}
  Moreover,
  \begin{align*}
& \E[e^{(\Re \beta)(Z_{(m,k]}(x) + Z_{(m,k]}(y))} \1_{A_{(m,k]}(x)}] \\
& \le \E[e^{(\Re \beta)(Z_{(m,k]}(x) + Z_{(m,k]}(y))} e^{(\alpha - \Re \beta) Z_{(m,k]}(x) - (\alpha - \Re \beta) \alpha \E Z_{(m,k]}(x)^2}] \\
& = \E[e^{\alpha Z_{(m,k]}(x) + (\Re \beta) Z_{(m,k]}(y) - \alpha^2 \E Z_{(m,k]}(x)^2 + \alpha (\Re \beta) \E Z_{(m,k]}(x)^2}] \\
& \lesssim e^{-\frac{\alpha^2}{2} \E Z_{(m,k]}(x)^2 + \alpha (\Re \beta) \E Z_{(m,k]}(x)^2 + \frac{(\Re \beta)^2}{2} \E Z_{(m,k]}(y)^2 + \alpha (\Re \beta) \E[Z_{(m,k]}(x) Z_{(m,k]}(y)]},
  \end{align*}
which together with the factor $e^{\frac{(\Im \beta)^2 - (\Re \beta)^2}{2} (\E Z_{(m,k]}(x)^2 + \E Z_{(m,k]}(y)^2)}$ gives us
  \begin{align*}
& \E[|E_{(m,k]}(x) \overline{E_{(m,k]}(y)}| \1_{A_{(m,k]}(x)}] \\
& \lesssim e^{- \frac{(\alpha - \Re \beta)^2}{2} \E Z_{(m,k]}(x)^2} e^{\frac{(\Im \beta)^2}{2} (\E Z_{(m,k]}(x)^2 + \E Z_{(m,k]}(y)^2) + \alpha (\Re \beta) \E[Z_{(m,k]}(x) Z_{(m,k]}(y)]},
  \end{align*}
  from which the claim follows by Lemma~\ref{lemma:z_covariance}.
\end{proof}

The following lemma is used in the proof of Proposition~\ref{prop:l2_estimate} below to handle the tail of the field for points $x,y \in I$ that are far enough from each other.

\begin{lemma}\label{lemma:l2_inner_estimate}
  Assume that $|x-y| \ge 2^{-m-1}$.
  Then for all $n \ge m$ we have
\[\1_{B_{m-1,m}(x)} \1_{B_{m-1,m}(y)} |\E[E_{(m,n]}(x) \overline{E_{(m,n]}(y)} \1_{B_{m,n}(x)} \1_{B_{m,n}(y)} | \F_m]| \lesssim 1.\]
\end{lemma}

\begin{proof}
  Define
\[P_k = E_{(m,k]}(x) \overline{E_{(m,k]}(y)} \1_{B_{m,k}(x)} \1_{B_{m,k}(y)}\]
  for $k \ge m$.
  Then
  \begin{align*}
P_{k+1} & = E_{(m,k+1]}(x) \overline{E_{(m,k+1]}(y)} \1_{B_{m,k}(x)} \1_{B_{m,k}(y)} \1_{B_{k,k+1}(x)} \1_{B_{k,k+1}(y)} \\
        & = -E_{(m,k+1]}(x) \overline{E_{(m,k+1]}(y)} \1_{B_{m,k}(x)} \1_{B_{m,k}(y)} (1  - \1_{B_{k,k+1}(x)} \1_{B_{k,k+1}(y)}) \\
              & \quad + \frac{e^{\beta Y_{k+1}(x)}}{\E e^{\beta Y_{k+1}(x)}} \cdot \frac{e^{\overline{\beta} Y_{k+1}(y)}}{\E e^{\overline{\beta} Y_{k+1}(y)}} P_k.
  \end{align*}
  Hence we have
  \begin{align}\label{eq:inductive_inequality}
    & \1_{\{Z_m(x) < \alpha \E Z_m(x)^2\}} \1_{\{Z_m(y) < \alpha \E Z_m(y)^2\}}|\E[P_{k+1} | \F_m]| \nonumber \\
     & \le \1_{\{Z_m(x) < \alpha \E Z_m(x)^2\}} \E[|E_{(m,k+1]}(x) \overline{E_{(m,k+1]}(y)}| \1_{A_{k+1}(x)} | \F_m] \nonumber \\
     & \quad + \1_{\{Z_m(y) < \alpha \E Z_m(y)^2\}} \E[|E_{(m,k+1]}(x) \overline{E_{(m,k+1]}(y)}| \1_{A_{k+1}(y)} | \F_m] \\
     & \quad + \frac{|\E e^{\beta Y_{k+1}(x) + \overline{\beta} Y_{k+1}(y)}|}{|\E e^{\beta Y_{k+1}(x)} \E e^{\overline{\beta} Y_{k+1}(y)}|} \1_{\{Z_m(x) < \alpha \E Z_m(x)^2\}} \1_{\{Z_m(y) < \alpha \E Z_m(y)^2\}}|\E[P_k | \F_m]|. \nonumber
  \end{align}
  Notice that $A_{k+1}(x) \cap \{Z_m(x) < \alpha \E Z_m(x)^2\} \subset A_{(m,k+1]}$.
  This, Lemma~\ref{lemma:tail_field_estimate} and Lemma~\ref{lemma:z_covariance} give us
  \begin{align*}
    & \1_{\{Z_m(x) < \alpha \E Z_m(x)^2\}} \E[|E_{(m,k+1]}(x) \overline{E_{(m,k+1]}(y)}| \1_{A_{k+1}(x)} | \F_m] \\
    & \le \E[|E_{(m,k+1]}(x) \overline{E_{(m,k+1]}(y)}| \1_{A_{(m,k+1]}(x)}] \\
    & \lesssim e^{- \frac{(\alpha - \Re \beta)^2}{2} \log(2) (k+1-m)} e^{(\Im \beta)^2 \log(2) (k+1-m) + \alpha (\Re \beta) \E[Z_{(m,k+1]}(x) Z_{(m,k+1]}(y)]} \\
    & \lesssim e^{(-\frac{(\alpha - \Re \beta)^2}{2} + (\Im \beta)^2) \log(2) (k+1-m)} =: \sigma_{k+1}. 
  \end{align*}
  Similarly we get
  \[\1_{\{Z_m(y) < \alpha \E Z_m(y)^2\}} \E[|E_{(m,k+1]}(x) \overline{E_{(m,k+1]}(y)}| \1_{A_{k+1}(y)} | \F_m] \lesssim \sigma_{k+1}\]
  and because of \eqref{eq:beta_imag_restriction} the $\sigma_k$ decay exponentially.
  Moving on to the third term on the right hand side of \eqref{eq:inductive_inequality}, let
  \begin{align*}
    \frac{|\E e^{\beta Y_{k+1}(x) + \overline{\beta} Y_{k+1}(y)}|}{|\E e^{\beta Y_{k+1}(x)} \E e^{\overline{\beta} Y_{k+1}(y)}|} =: \rho_{k+1}.
  \end{align*}
  By Lemma~\ref{lemma:z_covariance} there exists a constant $A$ such that for all $m \le m' \le M$ we have
  \begin{align*}
  \prod_{k=m'+1}^M \rho_k & = \frac{|\E e^{\beta Z_{(m',M]}(x) + \overline{\beta} Z_{(m',M]}(y)}|}{|\E e^{\beta Z_{(m',M]}(x)}| |\E e^{\overline{\beta} Z_{(m',M]}(y)}|} \\
                          & = \frac{|\E e^{\beta Z_M(x) + \overline{\beta} Z_M(y)}| |\E e^{\beta Z_{m'}(x)}| |\E e^{\overline{\beta} Z_{m'}(y)}|}{|\E e^{\beta Z_M(x)}| |\E e^{\overline{\beta} Z_M(y)}| |\E e^{\beta Z_{m'}(x) + \overline{\beta} Z_{m'}(y)}|} \\
                          & \le c e^{|\beta|^2 (\E Z_M(x) Z_M(y) - \E Z_{m'}(x) Z_{m'}(y))} \le A,
\end{align*}
  where $c > 0$ is a constant. We have thus verified that
  \begin{align*}
    & \1_{B_{m-1,m}(x)} \1_{B_{m-1,m}(y)} |\E[P_{k+1} | \F_m]| \\
    & \le C \sigma_{k+1} + \rho_{k+1} \1_{B_{m-1,m}(x)} \1_{B_{m-1,m}(y)} |\E[P_k | \F_m]|
  \end{align*}
  for some constant $C > 0$, and it is easy to see that we then have
  \[\1_{\{Z_m(x) < \alpha \E Z_m(x)^2\}} \1_{\{Z_m(y) < \alpha \E Z_m(y)^2\}} |\E[P_k | \F_m] \le C A (1 + \sum_{j=m+1}^\infty \sigma_j)\]
  for all $k \ge m$.
  Taking $k=n$ proves the claim.
\end{proof}

In the next lemma we bound the main contribution coming from points $x,y \in I$ that are at a given dyadic distance $2^{-m}$ from each other.

\begin{lemma}\label{lemma:l2_outer_estimate}
  Let $m \ge l+1$.
  If $|x-y| \le 2^{-m}$, we have
  \begin{align*}
  & \1_{A_l(x)} \1_{A_l(y)} \E[|E_{(l,m]}(x) \overline{E_{(l,m]}(y)}| \1_{B_{l,m}(x)} \1_{B_{l,m}(y)} | \F_l] \\
  & \lesssim e^{\big((\Re \beta)^2 - \frac{(2\Re \beta - \alpha)^2}{2}\big) \log(2) (m-l) + (\Im \beta)^2 \log(2) (m-l)}.
  \end{align*}
\end{lemma}

\begin{proof}
  By \eqref{eq:modulus_of_E_product} we have
  \[|E_{(l,m]}(x) \overline{E_{(l,m]}(y)}| \lesssim e^{(\Re \beta) (Z_{(l,m]}(x) + Z_{(l,m]}(y)) + ((\Im \beta)^2 - (\Re \beta)^2) \log(2) (m-l)}.\]
  On the other hand
  \begin{align*}
    \1_{A_l(x)} \1_{B_{l,m}(x)} & \le \1_{\{Z_l(x) \ge \alpha \E Z_l(x)^2\}} \1_{\{Z_m(x) < \alpha \E Z_m(x)^2\}} \\
                                & \le \1_{\{Z_{(l,m]}(x) < \alpha \E Z_{(l,m]}(x)^2\}} \\
                                & \le e^{(\Re \beta - \frac{\alpha}{2}) \alpha \E Z_{(l,m]}(x)^2 - (\Re \beta - \frac{\alpha}{2})Z_{(l,m]}(x)}
  \end{align*}
  and similarly for $\1_{A_l(y)} \1_{B_{l,m}(y)}$.
  By Lemma~\ref{lemma:z_covariance} we therefore have
  \begin{align*}
  & \1_{A_l(x)} \1_{A_l(y)} \E[|E_{(l,m]}(x) \overline{E_{(l,m]}(y)}| \1_{B_{l,m}(x)} \1_{B_{l,m}(y)} | \F_l] \\
  & \lesssim e^{((\Im \beta)^2 - (\Re \beta)^2) \log(2) (m-l) + (2 \Re \beta - \alpha) \alpha \log(2) (m-l)} \E[e^{\frac{\alpha}{2}(Z_{(l,m]}(x) + Z_{(l,m]}(y))}] \\
  & \lesssim e^{((\Im \beta)^2 - (\Re \beta)^2) \log(2) (m-l) + (2 \Re \beta - \alpha) \alpha \log(2) (m-l) + \frac{\alpha^2}{2} \log(2) (m-l)},
  \end{align*}
  from which the claim follows.
\end{proof}

The following proposition encodes our fundamental $L^2$-estimate.

\begin{proposition}\label{prop:l2_estimate}
  Let $I_{l,i}$ be a dyadic subcube of $I$ with diameter at most $2^{-l-1}$.
  Then for all $n \ge l$ and $f \in C(I,\complexes)$ we have
  \[\E\left[\left| \int_{I_{l,i}} f(x) E_n(x) \1_{A_l(x)} \1_{B_{l,n}(x)} \, dx \right|^2 \Big\vert \F_l \right] \lesssim 2^{-2dl} \|f\|^2_\infty \big(\sup_{x \in I_{l,i}} |E_l(x)|^2 \1_{A_l(x)}\big).\]
\end{proposition}

\begin{proof}
  Let us partition the set $I_{l,i}^2$ into sets $\mathcal{D}_m$, $l+1 \le m \le n$, by setting
  \[\mathcal{D}_m = \begin{cases}\{(x,y) \in I_{l,i}^2 : 2^{-m-1} \le |x-y| \le 2^{-m}\}, & \text{when } l+1 \le m \le n-1 \\
                                 \{(x,y) \in I_{l,i}^2 : |x-y| \le 2^{-n}\}, & \text{when } m=n.
                    \end{cases}\]
  Then one can write
  \begin{align*}
    & \E\left[\left| \int_{I_{l,i}} f(x) E_n(x) \1_{A_l(x)} \1_{B_{l,n}(x)} \, dx \right|^2 \Big\vert \F_l \right] \\
    & = \int_{I_{l,i}} \int_{I_{l,i}} f(x) \overline{f(y)} \E\big[E_n(x) \overline{E_n(y)} \1_{A_l(x)} \1_{A_l(y)} \1_{B_{l,n}(x)} \1_{B_{l,n}(y)} \big\vert \F_l \big] \, dx \, dy \\
    & = \int_{I_{l,i}} \int_{I_{l,i}} f(x) \overline{f(y)} E_l(x) \overline{E_l(y)} \1_{A_l(x)} \1_{A_l(y)} \times \\
    & \quad \E\big[E_{(l,n]}(x) \overline{E_{(l,n]}(y)} \1_{B_{l,n}(x)} \1_{B_{l,n}(y)} \big\vert \F_l \big] \, dx \, dy \\
    & = \sum_{m=l+1}^n \int_{(x,y) \in \mathcal{D}_m} f(x) \overline{f(y)} E_l(x) \overline{E_l(y)} \1_{A_l(x)} \1_{A_l(y)} \times \\
    & \quad \E\Big[E_{(l,m]}(x) \overline{E_{(l,m]}(y)} \1_{B_{l,m}(x)} \1_{B_{l,m}(y)} \times \\
    & \quad \E\big[E_{(m,n]}(x) \overline{E_{(m,n]}(y)} \1_{B_{m,n}(x)} \1_{B_{m,n}(y)} \big| \F_m\big] \Big| \F_l\Big] \\
    & \le \|f\|^2_\infty \big(\sup_{x \in I_{l,i}} |E_l(x)|^2 \1_{A_l(x)}\big) \times \\
    & \quad \sum_{m=l+1}^n \int_{(x,y) \in \mathcal{D}_m} \1_{A_l(x)} \1_{A_l(y)} \E\bigg[\big|E_{(l,m]}(x) \overline{E_{(l,m]}(y)}\big| \1_{B_{l,m}(x)} \1_{B_{l,m}(y)} \times \\
    & \quad \Big|\E\big[E_{(m,n]}(x) \overline{E_{(m,n]}(y)} \1_{B_{m,n}(x)} \1_{B_{m,n}(y)} \big| \F_m\big] \Big| \bigg\vert \F_l\bigg].
  \end{align*}
  Here we use the convention that $E_{(n,n]}(x) = E_{(n,n]}(y) = \1_{B_{n,n}(x)} = \1_{B_{n,n}(y)} = 1$.
  By Lemma~\ref{lemma:l2_inner_estimate}, Lemma~\ref{lemma:l2_outer_estimate}, and \eqref{eq:beta_imag_restriction} we have
  \begin{align*}
    & \sum_{m=l+1}^n \int_{(x,y) \in \mathcal{D}_m} \1_{A_l(x)} \1_{A_l(y)} \E\big[|E_{(l,m]}(x) \overline{E_{(l,m]}(y)}| \1_{B_{l,m}(x)} \1_{B_{l,m}(y)} \times \\
    & \quad \1_{B_{l,m}(x)} \1_{B_{l,m}(y)} |\E[E_{(m,n]}(x) \overline{E_{(m,n]}(y)} \1_{B_{m,n}(x)} \1_{B_{m,n}(y)} | \F_m]| \big\vert \F_l\big] \\
    & \lesssim  \sum_{m=l+1}^n \int_{(x,y) \in \mathcal{D}_m} e^{\big((\Re \beta)^2 - \frac{(2\Re \beta - \alpha)^2}{2}\big) \log(2) (m-l) + (\Im \beta)^2 \log(2) (m-l)} \\
    & \lesssim  \sum_{m=l+1}^n 2^{-md - ld} e^{\big((\Re \beta)^2 - \frac{(2\Re \beta - \alpha)^2}{2}\big) \log(2) (m-l) + (\Im \beta)^2 \log(2) (m-l)} \\
    & = 2^{-2ld} \sum_{j=1}^{n-l} e^{\big((\Re \beta)^2 - \frac{(2\Re \beta - \alpha)^2}{2}\big) \log(2) j + (\Im \beta)^2 \log(2) j - \log(2) j d} \\
    & \lesssim 2^{-2ld}. \qedhere
  \end{align*}
\end{proof}

Next we wish to bound the term $\sup_{x \in I_{l,i}} |E_l(x)| \1_{A_l(x)}$ that appears after Proposition~\ref{prop:l2_estimate} has been applied.
We come to our second main estimate.

\begin{proposition}\label{prop:sup_estimate}
  For sufficiently small $p > 1$ we have
  \[\E \sup_{x \in J} |E_l(x)|^p \1_{A_l(x)} \lesssim e^{-\varepsilon l}\]
  for some $\varepsilon > 0$.
  Here $J$ is a dyadic subcube of $I$ with side length proportional to $2^{-l}$.
  The estimate does not depend on $i$.
\end{proposition}

\begin{proof}
  We may assume that $p > 1$ is so small that $\alpha - p \Re \beta > 0$.
  Then
  \[|E_l(x)|^p \lesssim e^{p(\Re \beta) Z_l(x) - p \frac{(\Re \beta)^2 - (\Im \beta)^2}{2} \E Z_l(x)^2}\]
  and
  \[\1_{A_l(x)} \le e^{(\alpha - p \Re \beta) Z_l(x) - (\alpha - p \Re \beta) \alpha \E Z_l(x)^2},\]
  so
  \[\sup_{x \in J} |E_l(x)|^p \1_{A_l(x)} \lesssim e^{\alpha \sup_{x \in J} Z_l(x) - p \frac{(\Re \beta)^2 - (\Im \beta)^2}{2} \E Z_l(x)^2 - (\alpha - p \Re \beta) \alpha \E Z_l(x)^2}.\]
  Let $J$ have side length $2^{-\ell} \approx 2^{-l}$ and for all $m \ge 0$ let $D_m$ be the collection of dyadic subcubes of $J$ with side length $2^{-\ell-m}$.
  Choose a point $x_{m,i}$, $i=1,\dots,2^m$, from each cube in $D_m$.
  For example, one can take $x_{m,i}$ to be the center of its corresponding cube.
  Then for any $x \in J$ there exists a sequence of points $x_{m,i_m}$ converging to $x$ such that $x_{m,i_m}$ is chosen from inside the cube of $x_{m-1,i_{m-1}}$.
  Let $\pi(x_{m,i})$ be the point chosen from the parent cube of $x_{m,i}$.
  Now by continuity of $Z_l$ we have
  \begin{align*}
    e^{\alpha Z_l(x)} & \le e^{\alpha Z_l(x_{0,1})} + \sum_{m=1}^\infty |e^{\alpha Z_l(x_{m,i_m})} - e^{\alpha Z_l(x_{m-1,i_{m-1}})}| \\
                      & \le e^{\alpha Z_l(x_{0,1})} + \sum_{m=1}^\infty \sup_{i \in \{1,\dots,2^m\}} |e^{\alpha Z_l(x_{m,i})} - e^{\alpha Z_l(\pi(x_{m,i}))}|.
  \end{align*}
  Since the right hand side does not depend on $x$, we have
  \begin{align*}
    e^{\alpha \sup_{x \in J} Z_l(x)} & \le e^{\alpha Z_l(x_{0,1})} + \sum_{m=1}^\infty \sup_{i \in \{1, \dots, 2^m\}} |e^{\alpha Z_l(x_{m,i})} - e^{\alpha Z_l(\pi(x_{m,i}))}| \\
                                           & \le e^{\alpha Z_l(x_{0,1})} + \sum_{m=1}^\infty \Big( \sum_{i=1}^{2^m} |e^{\alpha Z_l(x_{m,i})} - e^{\alpha Z_l(\pi(x_{m,i}))}|^r \Big)^{1/r} \\
                                           & \le e^{\alpha Z_l(x_{0,1})} + \sum_{m=1}^\infty \Big( \sum_{i=1}^{2^m} |\alpha Z_l(x_{m,i}) - \alpha Z_l(\pi(x_{m,i}))|^r \times \\
        & \quad (e^{r \alpha Z_l(x_{m,i})} + e^{r \alpha Z_l(\pi(x_{m,i}))}) \Big)^{1/r},
  \end{align*}
  where $r > 1$.
  Taking the expectation and using Jensen's inequality gives
  \begin{align*}
    \E e^{\alpha \sup_{x \in J} Z_l(x)} & \le \E e^{\alpha Z_l(x_{0,1})} + \alpha \sum_{m=1}^\infty \Big( \sum_{i=1}^{2^m} \E\big[|Z_l(x_{m,i}) - Z_l(\pi(x_{m,i}))|^r \times \\
                                                  & \quad (e^{r \alpha Z_l(x_{m,i})} + e^{r \alpha Z_l(\pi(x_{m,i}))})\big] \Big)^{1/r}
  \end{align*}
  Moreover, by Hölder's inequality
  \begin{align*}
    & \E\big[|Z_l(x_{m,i}) - Z_l(\pi(x_{m,i}))|^r e^{r \alpha Z_l(x_{m,i})}] \\
    & \le \Big(\E |Z_l(x_{m,i}) - Z_l(\pi(x_{m,i}))|^{\frac{rs}{s-1}}\Big)^{\frac{s-1}{s}} \Big(\E e^{rs \alpha Z_l(x_{m,i})}\Big)^{1/s} \\
    & \lesssim \Big(\E |Z_l(x_{m,i}) - Z_l(\pi(x_{m,i}))|^{\frac{rs}{s-1}} \Big)^{\frac{s-1}{s}} e^{\frac{r^2 s \alpha^2}{2} \E Z_l(x_{m,i})^2},
  \end{align*}
  assuming that $r>1$ and $s>1$ are small.
  By condition \eqref{eq:continuity} we have
  \[\E |Z_j(x) - Z_j(y)|^q \le C_q 2^{jq} |x-y|^q\]
  for all $x,y \in I$ and $q \ge 2$ large enough, which gives us for $s$ close enough to $1$ that
  \begin{align*}
    \Big(\E |Z_l(x_{m,i}) - Z_l(\pi(x_{m,i}))|^{\frac{rs}{s-1}} \Big)^{\frac{s-1}{s}} \lesssim 2^{r l} |x_{m,i} - \pi(x_{m,i})|^{r} \lesssim 2^{-rm}.
  \end{align*}
  Hence we have
  \[\E\big[|Z_l(x_{m,i}) - Z_l(\pi(x_{m,i}))|^r e^{r \alpha Z_l(x_{m,i})}] \lesssim 2^{-rm} e^{\frac{r^2 s \alpha^2}{2} \log(2) l}.\]
  The same estimate holds also for $\E\big[|Z_l(x_{m,i}) - Z_l(\pi(x_{m,i}))|^r e^{r \alpha Z_l(\pi(x_{m,i}))}]$, so we get
  \begin{align*}
    \E e^{\alpha \sup_{x \in J} Z_l(x)} & \lesssim e^{\frac{\alpha^2}{2} \log(2) l} + \alpha \sum_{m=1}^\infty \Big( 2^m 2^{-rm} e^{\frac{r^2 s \alpha^2}{2} \log(2) l} \Big)^{1/r} \\
                                              & = e^{\frac{\alpha^2}{2} \log(2) l} + \alpha e^{\frac{rs \alpha^2}{2} \log(2) l} \sum_{m=1}^\infty 2^{m(\frac{1}{r} - 1)} \\
                                              & \lesssim e^{\frac{rs \alpha^2}{2} \log(2) l}.
  \end{align*}
  Putting everything together we have
  \[\E \sup_{x \in J} |E_l(x)|^p \1_{A_l(x)} \lesssim e^{\frac{rs \alpha^2}{2} \log(2) l - p \frac{(\Re \beta)^2 - (\Im \beta)^2}{2} \E Z_l(x)^2 - (\alpha - p \Re \beta) \alpha \E Z_l(x)^2},\]
  and by choosing $r$, $s$ and $p$ sufficiently close to $1$, we may make the exponent as close to
  \[\Big[-\frac{(\alpha - \Re \beta)^2}{2} + \frac{(\Im \beta)^2}{2}\Big] \log(2) l\]
  as we wish, which proves the claim.
\end{proof}

Having proved all the auxiliary results we need, we will now finish with the proof of Theorem~\ref{thm:main_result}.
Let $n \ge 1$ and set $N = t_n - 1$.
We may then write $\mu_N(f;\beta)$ as the sum
\[\mu_N(f;\beta) = \sum_{l=0}^n \int_I f(x) E_n(x) \1_{A_l(x)} \1_{B_{l,n}(x)} \, dx.\]
Here the $l$th term of the sum contains the contribution from those points for which the last time the field is exceptionally large is $l$.
By Minkowski's and Jensen's inequalities it follows that for $p \in (1,2)$ we have
\begin{align*}
  & \Vert \mu_N(f;\beta)\Vert_{L^p(\Omega)} \\
  & \le \sum_{l=0}^n \sum_{i=0}^{c 2^{d l}}\left\Vert \int_{I_{l,i}} f(x) E_n(x) \1_{A_l(x)} \1_{B_{l,n}(x)} \, dx \right\Vert_{L^p(\Omega)} \\
  & \le \sum_{l=0}^n \sum_{i=0}^{c 2^{d l}} \left(\E\left[ \left(\E\left[\left|\int_{I_{l,i}} f(x) E_n(x) \1_{A_l(x)} \1_{B_{l,n}(x)} \, dx\right|^2 \Big\vert \F_l\right] \right)^{p/2} \right]\right)^{1/p},
\end{align*}
where $c > 0$ is a constant and for fixed $l$ the sets $I_{l,i} \subset I$ are dyadic subcubes of $I$ with side length $2^{-\ell}$, $\ell = l + c'$, where $c' > 0$ is a constant depending on $d$ that ensures that the diameter of $I_{l,i}$ which has length $\sqrt{d} 2^{-\ell}$ is at most $2^{-l-1}$.
By Proposition~\ref{prop:l2_estimate} and Proposition~\ref{prop:sup_estimate} we obtain for small enough $p$ that
\begin{align*}
  \Vert \mu_N(f;\beta)\Vert_{L^p(\Omega)} & \lesssim \|f\|_\infty \sum_{l=0}^n \sum_{i=0}^{c 2^{d l}} 2^{-dl} \big(\E[\sup_{x \in I_{l,i}} |E_l(x)|^p \1_{A_l(x)}]\big)^{1/p} \\
                                          & \lesssim \|f\|_\infty \sum_{l=0}^n \sum_{i=0}^{c 2^{dl}} 2^{-dl} e^{-\varepsilon l / p} \lesssim \|f\|_\infty.
\end{align*}
Thus we have proven the $L^p$ boundedness along the subsequence $\mu_{t_n - 1}$, and since the $L^p$ norms of a martingale are increasing, it follows that the whole sequence $\mu_n$ has bounded $L^p$ norms.
It is clear that the $p>1$ for which we get the bound depends continuously on $\beta$.
We obtain thus an open set $U$ containing $(0,\sqrt{2d})$ such that for any compact $K \subset U$ there exists $p > 1$ for which we have the uniform $L^p(\Omega)$-boundedness of $\mu_n(f;\beta)$, $\beta \in K$.

To get rid of the assumption that $\delta > \sqrt{d}$, we can partition $I$ into a finite number of dyadic cubes $I_k$, $1 \le k \le m$, with diameter less than $\delta$.
By Minkowski's inequality we have
\[\Vert\mu_n(f;\beta)\Vert_{L^p(\Omega)} \le \sum_{k=1}^m \Vert \mu_n(f\chi_{I_k};\beta) \Vert_{L^p(\Omega)}\]
and the above proof works for every summand, which yields the result in the case of general $\delta$.

Finally, let $B_{2r} = B(\beta_0, 2r) \subset U$ be an open ball of radius $2r > 0$ and let $p>1$ be such that $\E |\mu_n(f;z)|^p$ is uniformly bounded both in $z \in B_r$ and $n \ge 1$.
Then almost surely for all $n \ge 1$ the function $z \mapsto \mu_n(f;z)$ belongs to the Bergman space $A^p(B_r) = L^p(B_r) \cap H(B_r)$, where $H(B_r)$ is the space of analytic functions on $B_r$.
As a closed subspace of $L^p(B_r)$ the space $A^p(B_r)$ has the Radon--Nikodym property.
Hence the uniform boundedness
\[\E \|\mu_n(f;\cdot)\|_{A^p(B_r)}^p = \int_{B_r} \E |\mu_n(f;x + i y)|^p \, dx \, dy \le |B_r| \sup_{n \ge 1} \sup_{z \in B_r} \E |\mu_n(f;z)|^p < \infty\]
for $n \ge 1$ of the Bergman norm in $L^p(\Omega)$ implies that the $A^p(B_r)$-valued martingale $\mu_n(f;z)$ converges almost surely in $A^p(B_r)$.
In particular we have uniform convergence on all compact subsets of $B_r$ to an analytic function $z \mapsto \mu(f;z)$.
This finishes the proof of Theorem~\ref{thm:main_result}.

\section{Examples}

In this section we consider two basic examples of non-Gaussian chaos measures for which our theory applies.
In the end we also discuss some open questions.

\subsection{The non-Gaussian Fourier series}

As our first application of Theorem~\ref{thm:main_result} we will prove Theorem~\ref{thm:fourier_result}.
Recall that we are interested in the random Fourier series
\[\sum_{k=1}^\infty X_k(x) = \sum_{k=1}^\infty \frac{1}{\sqrt{k}}(A_k \cos(2\pi k x) + B_k \sin(2\pi k x)),\]
where $A_k$ and $B_k$ are i.i.d. random variables satisfying $\E e^{\lambda A_1} < \infty$ for all $\lambda \in \reals$.

The assumptions of Theorem~\ref{thm:main_result} are fairly straightforward to establish.

\begin{proof}[Proof of Theorem~\ref{thm:fourier_result}]
  The moment condition \eqref{eq:moment_condition} is clear.
  The derivative of $X_k$ in this case is
  \[X_k'(x) = 2\pi\sqrt{k}(-A_k \sin(2\pi k x) + B_k \cos(2\pi k x)),\]
  and it is easy to check that its supremum is
  \[2\pi\sqrt{k}\sqrt{A_k^2 + B_k^2},\]
  whose $r$-th moment is of order $k^{r/2}$.
  Hence there exists a constant $C > 0$ such that
  \[\sum_{k=1}^n \E |\sup_{x \in I} X_k'(x)|^r \le C n^{r/2 + 1}.\]
  On the other hand $e^{r \sum_{k=1}^n \E X_k(0)^2} = e^{r \log(n) + O(1)}$ is of order $n^r$ so \eqref{eq:continuity3} clearly holds.

  It remains to verify that the field has a locally log-correlated structure.
  The condition $\E X_k(x)^2 \to 0$ obviously holds.
  In \eqref{eq:log_correlated} we choose $\delta = \frac{1}{2}$.
  Notice that $\sum_{k=1}^n \E X_k(0)^2 = \log(n) + O(1)$.
  Assume first that $|x-y| \le 1/n$.
  Then
  \[\min(\log \frac{1}{|x-y|}, \sum_{k=1}^n \E X_k(0)^2) = \log(n) + O(1).\]
  On the other hand
  \begin{align*}
    \sum_{k=1}^n \E X_k(x) X_k(y) & = \sum_{k=1}^n \frac{\cos(2\pi k (x-y))}{k} \\
                                  & = \sum_{k=1}^n \frac{1 + O(k^2 (x-y)^2)}{k} = \log(n) + O(1),
  \end{align*}
  which shows that \eqref{eq:log_correlated} holds in this case.
  Assume then that $\frac{1}{2} \ge |x-y| > 1/n$.
  In this case
  \begin{align*}
    & \min(\log \frac{1}{|x-y|}, \sum_{k=1}^n \E X_k(0)^2) = \log \frac{1}{|x-y|} + O(1) = \log \frac{1}{2\sin(\pi |x-y|)} + O(1) \\
    & = \sum_{k=1}^n \frac{\cos(2\pi k (x-y))}{k} + \sum_{k=n+1}^\infty \frac{\cos(2\pi k (x-y))}{k} + O(1),
  \end{align*}
  and it is enough to show that
  \[\sum_{k=n+1}^\infty \frac{\cos(2\pi k (x-y))}{k}\]
  is bounded.
  Let
  \[R_n(t)= \sum_{k=n+1}^\infty \frac{\cos(2\pi k t)}{k} = \Re \sum_{k=n+1}^\infty \frac{e^{2\pi i k t}}{k}.\]
  For all $N \ge n+1$ we have
  \[\sum_{k=n+1}^N \frac{e^{2\pi i k t}}{k} = \sum_{k=n+1}^N e^{2\pi i k t} \int_0^1 r^{k-1} \, dr = \int_0^1 \frac{e^{2\pi i (N+1) t} r^{N} - e^{2\pi i (n+1) t} r^{n}}{e^{2\pi i t} r - 1} \, dr,\]
  and since we have the inequalities
  \[|e^{2\pi i t} r - 1| \ge \begin{cases} 1, & \text{when } \frac{1}{2} \ge |t| \ge \frac{1}{4},\\
  \sin(2\pi t), & \text{when } |t| \le \frac{1}{4},
    \end{cases}\]
  we see that for fixed $|t| > 0$ the denominator is bounded and hence we may take limit as $N \to \infty$ to get
  \[|R_n(t)| \le \int_0^1 \frac{r^n}{|e^{2\pi i t} r - 1|} \, dr.\]
  If $\frac{1}{2} \ge |t| \ge \frac{1}{4}$, we have $|R_n(t)| \le \int_0^1 r^n \, dr = \frac{1}{n+1}$ and similarly if $0 < |t| \le \frac{1}{4}$, we have $|R_n(t)| \le \frac{1}{(n+1) \sin(2\pi t)}$. 
  This shows the boundedness of $|R_n(x-y)|$ when $\frac{1}{2} \ge |x-y| \ge \frac{1}{n}$.
\end{proof}

\subsection{Chaos induced by dilations of a stationary process}

Our second example is a small variation of the one given by Mannersalo, Norros, and Riedi in \cite{mannersalo2002multifractal}.
In their paper random measures of the form
\[\prod \Lambda_k(b^k x) \, dx\]
were studied on the real line.
Here $\Lambda_k$ are non-negative stationary i.i.d. random processes with expectation $1$, $dx$ is the $1$-dimensional Lebesgue measure, and $b > 1$ is a constant.

In the present example we define
\begin{equation}\label{eq:X_k_dilated}
X_k(x) = a_k Y_k(b_k x),
\end{equation}
where $Y_k$ are i.i.d. continuous, centered and stationary random fields on $\reals^d$ with the covariance function
\[\E Y_k(x) Y_k(y) = f(|x-y|)\]
and $a_k,b_k > 0$ are constants.
Thus the case $a_k = 1$, $b_k = b^k$ would allow us to write $\Lambda_k(x) = \frac{e^{Y_k(b^k x)}}{\E e^{Y_k(b^k x)}}$ and obtain the situation of \cite{mannersalo2002multifractal}.
We are not, however, quite able to consider this particular case since our method requires weak decay from the coefficients $a_k$.

First of all, we assume that
\begin{equation}\label{eq:Y_exponential_assumption}
  \E e^{\lambda Y_k(0)} < \infty
\end{equation}
for all $\lambda \in \reals$ and that $Y_k(x)$ satisfies the regularity condition
\begin{equation}\label{eq:Y_regularity}
  \E |Y_k(x) - Y_k(y)|^r \le C_r |x-y|^r \quad \text{ for all } r \ge 2,
\end{equation}
where $C_r > 0$ is a constant that may depend on $r$.
Furthermore we require that
\begin{equation}\label{eq:f_assumption}
  f(t) = 1 + O(t) \quad \text{and} \quad f(t) = O(t^{-\delta})
\end{equation}
for some $\delta > 0$.

As a side remark, notice that if \eqref{eq:Y_regularity} holds for some sequence $r_n \to \infty$, then by Hölder's inequality it holds for all $r \ge 2$.

For this model we obtain the following result.

\begin{theorem}\label{thm:general_scaling}
  Let $a_k > 0$ be such that
  \[\lim_{k \to \infty} a_k = 0, \quad \sum_{k=1}^\infty a_k^2 = \infty, \quad \text{ and } \quad \sum_{k=1}^\infty a_k^3 < \infty.\]
  Assume that $b_k$ are of the form
  \[b_k = \exp(\sum_{j=1}^k a_k^2 + c_k),\]
  where $c_k \in \reals$ satisfy $\sup_{k \ge 1} |c_k| < \infty$.
  Then the conclusions of Theorem~\ref{thm:fourier_result} hold for the chaos measures
  \[d \mu_n = \frac{e^{\sum_{k=1}^n a_k Y_k(b_k x)}}{\E e^{\sum_{k=1}^n a_k Y_k(b_k x)}} \, dx,\]
  where $Y_k$ is assumed to satisfy \eqref{eq:Y_exponential_assumption}, \eqref{eq:Y_regularity} and \eqref{eq:f_assumption}.
\end{theorem}

Towards the proof of the above result, we define
\[A_n = \sum_{k=1}^n \E X_k(0)^2 = \sum_{k=1}^n a_k^2, \quad n \ge 1.\]
We then have the following lemma.

\begin{lemma}\label{lemma:exp_sum}
  Let $\alpha \in \reals$, $\alpha \neq 0$.
  Then there exists $C_\alpha > 0$ such that
  \[\sum_{k=m}^n a_k^2 e^{\alpha A_k} \le C_{\alpha} |e^{\alpha A_n} - e^{\alpha A_{m-1}}|\]
  for all $n \ge m \ge 1$.
\end{lemma}

\begin{proof}
  Assume first that $\alpha > 0$.
  We have
  \[\sum_{k=m}^n a_k^2 e^{\alpha A_k} = \sum_{k=m}^n (A_k - A_{k-1}) e^{\alpha A_k}.\]
  Let $S = \sup_{k \ge 1} a_k^2$.
  Since $a_k \to 0$, there exists $k_0 \ge 1$ such that for $k \ge k_0$ we have $A_k - A_{k-1} \le 1$.
  The elementary inequality $x \le \frac{S}{1 - e^{-\alpha S}} (1 - e^{-\alpha x})$ for $0 \le x \le S$ gives us
  \[\sum_{k=m}^n a_k^2 e^{\alpha A_k} \le C_\alpha \sum_{k=m}^n (e^{\alpha A_k} - e^{\alpha A_{k-1}}) = C_\alpha (e^{\alpha A_n} - e^{\alpha A_{m-1}}),\]
  where $C_\alpha = \frac{S}{1 - e^{-\alpha S}}$.

  On the other hand if $\alpha < 0$, then we have the inequality $x \le \frac{-1}{\alpha} (e^{-\alpha x} - 1)$ for $x \ge 0$.
  Thus
  \[\sum_{k=m}^n a_k^2 e^{\alpha A_k} \le C_\alpha \sum_{k=m}^n (e^{\alpha A_{k-1}} - e^{\alpha A_k}) = C_\alpha (e^{\alpha A_{m-1}} - e^{\alpha A_{n}}),\]
  where $C_\alpha = \frac{-1}{\alpha}$.
\end{proof}

\begin{proof}[Proof of Theorem~\ref{thm:general_scaling}]
  Once again, it is enough to check the assumptions of Theorem~\ref{thm:main_result}.
  We will start by showing that we have a locally log-correlated structure.
  For this it is enough to verify that \eqref{eq:log_correlated} holds, the other conditions being trivial.
  Assume first that $n \ge 1$ and $x,y \in I$ are such that $A_n < \log \frac{1}{|x-y|}$ holds.
  Then by Lemma~\ref{lemma:exp_sum} the left hand side of \eqref{eq:log_correlated} satisfies the inequality
  \begin{align}
    \Big|\sum_{k=1}^n \E X_k(x) X_k(y) - A_n\Big| & \le \sum_{k=1}^n a_k^2 |f(b_k |x-y|) - 1| \lesssim \sum_{k=1}^n a_k^2 b_k |x-y| \nonumber \\
                                                  & \lesssim \sum_{k=1}^n a_k^2 e^{A_k} |x-y| \lesssim (1 + e^{A_n}) |x-y| \lesssim 1. \label{eq:dilated_series_head}
  \end{align}
  Thus \eqref{eq:log_correlated} holds in this case.
  Assume then that $A_n \ge \log \frac{1}{|x-y|}$ and let $N$ be the last index for which $A_N < \log \frac{1}{|x-y|}$.
  In this case we have
  \begin{align*}
    \sum_{k=1}^n \E X_k(x) X_k(y) - \log \frac{1}{|x-y|} & = [\sum_{k=1}^{N} a_k^2 f(b_k |x-y|) - A_N] \\
    & \quad + [A_N - \log \frac{1}{|x-y|}] + \sum_{k=N+1}^n a_k^2 f(b_k |x-y|).
  \end{align*}
  The first term on the right hand side is bounded by \eqref{eq:dilated_series_head}  and the second term is bounded simply by the way $A_N$ was chosen.
  Finally, again by Lemma~\ref{lemma:exp_sum}, we have
  \[\sum_{k=N+1}^n a_k^2 |f(b_k |x-y|)| \lesssim \sum_{k=N+1}^n a_k^2 b_k^{-\delta} |x-y|^{-\delta} \lesssim |x-y|^{-\delta} e^{-\delta A_N}.\]
  Since $e^{-\delta A_N} \approx |x-y|^{\delta}$, we see that \eqref{eq:log_correlated} holds also in this case.

  Next we will verify the regularity conditions \eqref{eq:moment_condition} and \eqref{eq:continuity}.
  Clearly the requirement $\sup_{x \in I} \E e^{\lambda X_k(x)} < \infty$ holds by assumption \eqref{eq:Y_exponential_assumption}.
  Moreover, we have
  \[\Big(\E |X_k(x)|^{3 + \varepsilon}\Big)^{\frac{3}{3+\varepsilon}} = a_k^{3} \Big(\E |Y_k(0)|^{3 + \varepsilon}\Big)^{\frac{3}{3+\varepsilon}},\]
  and thus \eqref{eq:moment_condition} holds.
  Finally, we have 
  \[\sum_{k=1}^n \E|X_k(x) - X_k(y)|^r \le C_r \sum_{k=1}^n a_k^r b_k^r |x-y|^r \lesssim C_r b_n^r |x-y|^r \lesssim C_r e^{r \sum_{k=1}^n a_k^2} |x-y|^r\]
  by Lemma~\ref{lemma:exp_sum}, so \eqref{eq:continuity2} holds and therefore also the condition \eqref{eq:continuity} follows.
\end{proof}

To illustrate the relationship between $a_k$ and $b_k$ in Theorem~\ref{thm:general_scaling}, we mention the following corollary.

\begin{corollary}
  Theorem~\ref{thm:general_scaling} holds in the following three cases:
  \begin{itemize}
    \item $a_k = k^{-\alpha}$ for some $\frac{1}{3} < \alpha < \frac{1}{2}$ and $b_k = e^{\frac{k^{1-2\alpha}}{1-2\alpha}}$,
    \item $a_k = k^{-1/2}$ and $b_k = k$, or
    \item $a_k = (k \log k)^{-1/2}$ and $b_k = \log k$ when $k \ge 2$ and $a_1 = b_1 = 0$.
  \end{itemize}
\end{corollary}

As an example of stationary process for which our theorem is valid one can take a stationary Gaussian process $Y(x)$ whose covariance function $f$ satisfies $f(t) = 1 + O(t^2)$ and decays like $t^{-\delta}$ for some $\delta > 0$.
Indeed, in this case we have
\[\E |Y(x) - Y(y)|^r \lesssim_r (\E |Y(x) - Y(y)|^2)^{\frac{r}{2}} = 2^{\frac{r}{2}} (1 - f(|x-y|))^{\frac{r}{2}} \lesssim_r |x-y|^r.\]

A very simple non-Gaussian example is given by $Y(x) = A \cos(x + U)$, where $A$ and $U$ are independent, $U$ is uniformly distributed in $[0,2\pi)$, $\E A = 0 $, $\E A^2 = 2$ and $\E e^{\lambda A} < \infty$ for all $\lambda \in \reals$. It is easy to construct more complicated families of non-Gaussian stationary processes satisfying the conditions of the theorem.

\subsection{Open questions}

On the foundational level several questions are widely open in the case of non-Gaussian chaos.
An important property of Gaussian chaos is universality, i.e. independence of the chaos of the approximation used.
For results in this direction, see e.g. \cite{robert2010gaussian,shamov2014gaussian,junnila2015uniqueness}.
No such general results are known in the non-Gaussian case.
Another interesting open problem is the construction of critical chaoses.

Moreover, it would be interesting to study the finer properties of the resulting chaoses.
For example, can one determine the multifractal spectrum of the measure?
Is it possible to determine the exact $L^p$-integrability or the tail behaviour of the total mass?
In general one may try to examine what are the differences and similarities of these measures to their Gaussian counterparts.

\end{document}